\documentclass[USenglish,onecolumn]{article}

\usepackage[utf8]{inputenc}

\usepackage[left=3cm,right=3cm,top=3cm,bottom=3cm]{geometry} 
\usepackage{amsmath}
\usepackage{mathtools} 
\usepackage{easybmat}  
\usepackage{arydshln}  
\usepackage{enumitem}  
\usepackage[colorlinks,citecolor=blue,linkcolor=blue,urlcolor=black,bookmarks=false,hypertexnames=true]{hyperref} 

\usepackage{amsthm}
\newtheorem{thm}{Theorem}
\theoremstyle{plain}
\newtheorem{prop}[thm]{Proposition}
\theoremstyle{plain}

\theoremstyle{remark}
\newtheorem*{rem*}{Remark}
\theoremstyle{plain}
\newtheorem{lem}[thm]{Lemma}
\theoremstyle{definition}

\theoremstyle{definition}

\theoremstyle{plain}
\newtheorem*{conjecture*}{Conjecture}

\let\Re\relax
\DeclareMathOperator{\Re}{Re}

\usepackage{xcolor}

\begin{document}

\title{The Perron non-backtracking eigenvalue after node addition}

\author{
Leo Torres \\
Max Planck Institute for Mathematics in the Sciences \\
\url{leo@leotrs.com}
}

\maketitle

\begin{abstract}{
Consider a finite undirected unweighted graph $G$ and add a new node to it arbitrarily connecting it to pre-existing nodes.
We study the behavior of the Perron eigenvalue of the non-backtracking matrix of $G$ before and after such a node addition.
We prove an interlacing-type result for said eigenvalue, namely, the Perron eigenvalue never decreases after node addition.
Furthermore, our methods lead to bounds for the difference between the eigenvalue before and after node addition.
These are the first known bounds that have been established in full rigor.
Our results depend on the assumption of diagonalizability of the non-backtracking matrix.
Practical experience says that this assumption is fairly mild in many families of graphs, though necessary and sufficient conditions for it remain an open question.


}
\end{abstract}

\section{Introduction}\label{sec:intro}

Given a finite, undirected, unweighted graph $G=(V,E)$, let $\bar{E}$ be the set of oriented edges.
That is, if there exists an edge in $E$ joining nodes $u,v \in V$, then both orientations $u \to v$ and $v \to u$ are members of $\bar{E}$.
A \emph{walk} in $G$ is a finite sequence of oriented edges that are consecutively adjacent.
A \emph{backtrack} is a walk of length two of the type $u \to v \to u$.
A \emph{non-backtracking walk} is a walk such that none of its sub-walks are backtracks.
The \emph{non-backtracking matrix} of $G$ is the (unnormalized) transition matrix of a random walker that only follows non-backtracking walks.
In the rest of the paper, we use the initials ``NB'' to mean ``non-backtracking'' and we say ``NB eigenvalues'' to mean ``the eigenvalues of the NB matrix''.

The NB matrix and its eigenvalues have received much attention recently.
In many applications to network science such as spectral clustering \cite{krzakala,bordenave2015}, centrality \cite{newman,torres2020}, and dynamics \cite{karrer,shrestha,pastor}, the NB eigenvalues improve upon what can be achieved with the spectrum of other matrices associated to the graph such as the adjacency or the Laplacian.
However, in contrast to those other matrices, the NB matrix is not normal and thus not symmetric.
Therefore, most of the standard tools in spectral graph theory do not apply to it.

One such tool that is absent in the NB case is a theory of interlacing \cite{godsil}.
In this paper, we establish an interlacing-type result for one of the NB eigenvalues.
Perron-Frobenius theory guarantees that the NB matrix has a real, positive, simple eigenvalue equal to its spectral radius.
We refer to it as the Perron eigenvalue.
Our main result establishes that the Perron eigenvalue behaves as expected when a node is added to or removed from the graph: it never decreases when a node is added, and it never increases when a node is removed.
Concretely, suppose $\lambda_1$ is the Perron eigenvalue before node addition and $\lambda_c$ is the Perron eigenvalue after node addition. 
We show $\lambda_c \geq \lambda_1$, with equality if and only if the degree of the added node is $1$ and, furthermore, we provide some bounds for the difference $\lambda_c - \lambda_1$.
It should be noted that establishing $\lambda_c \geq \lambda_1$ can be done with other, more elementary, methods than the one presented here (see e.g. Example 7.10.2 in \cite{meyer}).
However, our method is the only known method that leads to reasonable bounds.

The present work is a direct continuation of the work in \cite{torres2020}. Therein, the authors use heuristics to bound the difference $\lambda_{c}-\lambda_1$, and develop algorithms exploiting these heuristics to maximize the difference in an applied setting.
Here we improve upon those heuristics and establish them in full rigor.
Other authors \cite{zhang,zhang2} have studied the analogous case when an edge, rather than a node, is added to or removed from the graph.

Crucially, our derivation rests on the assumption that the NB matrix is diagonalizable.
Virtually every graph observed during this and previous studies has a diagonalizable NB matrix.
Thus the assumption of diagonalizability seems to be reasonably mild.
However, characterizing the conditions under which this assumption is true remains a work in progress \cite{torresarxiv}.

The rest of this paper is structured as follows.
In Section \ref{sec:background} we define the necessary terms and background notions.
In Section \ref{sec:diagonal} we introduce the assumption of diagonalizability and discuss some of its consequences.
In particular, Equation \eqref{eqn:assumption} presents a novel property of the eigendecomposition of the NB matrix that is of independent interest. 
In Section \ref{sec:addnode} we describe the setting of node addition and recall some of the results already published in \cite{torres2020}.
In Section \ref{sec:interlacing} we prove our main result of interlacing of the Perron eigenvalue and in Section \ref{sec:bounds} we introduce the advertised bounds.
We close with our conclusions in Section \ref{sec:conclusion}.

\vspace{-1em}
\section{Background}\label{sec:background}

All graphs considered are finite, undirected, unweighted, and connected.
The symbol $\| \cdot \|_p$ denotes the induced $p$-norm of a matrix.
In this work, we use ``NB'' to mean ``non-backtracking''.

\paragraph*{NB walks and NB matrix}
Consider a graph $G$ with $n$ nodes and $m$ edges.
Let $\bar{E} $ be the set of \emph{oriented} edges.
That is, if nodes $u$ and $v$ are neighbors in $G$, we let both $u \to v$ and $v \to u$ be members of $\bar{E}$.
A \emph{walk} in $G$ is a sequence of oriented edges that are consecutively adjacent.
The \emph{length} of a walk is the number of oriented edges in it.
A walk of length two of the type $u \to v \to u$ is called a \emph{backtrack}.
A \emph{non-backtracking} walk, or NB walk, is a walk such that none of its length-$2$ sub-walks comprise a backtrack.
The \emph{NB matrix} of $G$ is denoted by $\mathbf{B}$ and can be understood as the (unnormalized) transition matrix of a random walker that does not trace backtracks.
Formally, it is indexed in the rows and columns by $\bar{E}$, and it is defined as
\begin{equation}\label{eqn:def-b}
\mathbf{B}_{k \to l, i \to j} = \delta_{jk} \left( 1 - \delta_{il} \right),
\end{equation}
where $\delta$ is the Kronecker delta.
In words, $\mathbf{B}_{k \to l, i \to j}$ is one whenever $k$ equals $j$ and the walk $i \to j = k \to l$ is a valid walk of length 2 that is not a backtrack, and it is zero otherwise.
Importantly, the quantity $\mathbf{B}^r_{k \to l, i \to j}$ is equal to the number of NB walks starting with $i \to j$ and ending with $k \to l$ of length $r + 1$.

\paragraph*{NB eigenvalues} 
The NB matrix $\mathbf{B}$ is not normal (and thus not symmetric); its eigenvalues are in general complex numbers.
Therefore, most of the standard tools from spectral graph theory do not apply to it.
An important fact about NB eigenvalues is the \emph{Ihara determinant formula} \cite{friedman,kotani}.
Let $\mathbf{A} = \left( a_{ij} \right)$ be the adjacency matrix of $G$, and let $\mathbf{D}$ the diagonal degree matrix.
Then, the Ihara determinant formula states
\begin{equation}\label{eqn:ihara}
\det \left( I - t \mathbf{B} \right) = \left( 1 - t^2 \right)^{m - n} \det \left( I - t \mathbf{A} + t^2 \left( \mathbf{D} - \mathbf{I} \right) \right).
\end{equation}
Note the left-hand side of Equation \eqref{eqn:ihara} is the characteristic polynomial of $\mathbf{B}$ evaluated at $1/t$.

Important in our later exposition are the eigenvalues $\pm1$.
Equation \eqref{eqn:ihara} implies that $\pm 1$ are always eigenvalues of $\mathbf{B}$, each with algebraic multiplicity at least $m - n$.
Analyzing the equation $\mathbf{B} \mathbf{v} = \pm \mathbf{v}$ using the definition \eqref{eqn:def-b} reveals
\begin{equation}\label{eqn:pm1}
\small
\mathbf{B} \mathbf{v} = \mathbf{v} \iff
\begin{cases}
\sum_i a_{ij} \mathbf{v}_{i \to j} & \text{for each node $j$,} \\
\mathbf{v}_{k \to l} + \mathbf{v}_{l \to k} = 0 & \text{for each edge $k - l$.}
\end{cases}
\quad\quad
\mathbf{B} \mathbf{v} = - \mathbf{v} \iff
\begin{cases}
\sum_i a_{ij} \mathbf{v}_{i \to j} & \text{for each node $j$,} \\
\mathbf{v}_{k \to l} = \mathbf{v}_{l \to k} & \text{for each edge $k - l$.}
\end{cases}
\end{equation}

\section{Diagonalizability}\label{sec:diagonal}
In general, neither sufficient or necessary conditions for the diagonalizability of $\mathbf{B}$ are known. However, the assumption of diagonalizability is crucial in our later exposition. We remark here that the vast majority of graphs generated with certain random graph ensembles (e.g Erd\H{o}s-R\'enyi, Barab\'asi-Albert, Stochastic Block Model) seem to have a diagonalizable NB matrix when the number of nodes is large.

To discuss the diagonalizability of $\mathbf{B}$, consider the orientation reversal operator $\mathbf{P}$, defined by
\begin{equation}\label{eqn:def-p}
\left( \mathbf{P} \mathbf{v} \right)_{k \to l} = \mathbf{v}_{l \to k},
\end{equation}
It is direct from the definition that $\mathbf{P}$ is involutory ($\mathbf{P}^2 = \mathbf{I}$), orthogonal ($\mathbf{P P}^{\top} = \mathbf{I}$), and symmetric ($\mathbf{P} = \mathbf{P}^{\top}$).
This operator reveals the structure of the NB matrix $\mathbf{B}$ in the following sense.

\begin{lem}[see \cite{bordenave2015,torres2020}]\label{lem:left-right}
Let $G$ be a graph with NB matrix $\mathbf{B}$ and let $\mathbf{P}$ be the orientation reversal operator defined in Equation \eqref{eqn:def-p}. Then, we have
\vspace{-1em}
\begin{enumerate}
\item $\mathbf{B}$ is $\mathbf{P}$-symmetric, that is $\mathbf{B}^{\top} = \mathbf{P B P}$.
\item The right and left eigenvectors of $\mathbf{B}$ are related via $\mathbf{P}$, that is $\mathbf{B v} = \lambda \mathbf{v} \iff \mathbf{v}^{\top} \mathbf{P B} = \lambda \mathbf{v}^{\top} \mathbf{P}$.
\item Any two right eigenvectors $\mathbf{B v} = \lambda \mathbf{v}$ and $\mathbf{B u} = \mu \mathbf{u}$ such that $\mu \neq \lambda$ are $\mathbf{P}$-orthogonal, that is $\mathbf{v}^{\top} \mathbf{P u} = 0$.
\end{enumerate}
\end{lem}

\begin{proof}
The first statement is direct from Equations \eqref{eqn:def-b} and \eqref{eqn:def-p}, and implies the second statement.
The third statement comes from the second statement and the fact that, for any matrix, any left eigenvector is orthogonal to a right eigenvector of different eigenvalue.
\end{proof}

Now assume $\mathbf{B}$ is diagonalizable, that is, there exists an invertible matrix $\mathbf{R}$ such that $\mathbf{B R} = \mathbf{R \Lambda}$, where $\mathbf{\Lambda}$ is a diagonal matrix containing the eigenvalues of $\mathbf{B}$.
Define $\mathbf{L} = \mathbf{R}^{-1}$ and note the rows of $\mathbf{L}$ form a basis of left eigenvectors of $\mathbf{B}$.
In light of statement 2 of Lemma \ref{lem:left-right}, the rows of the matrix $\mathbf{R}^{\top} \mathbf{P}$ also form a basis of left eigenvectors.
Thus we ask whether $\mathbf{R}$ can be chosen so that
\begin{equation}\label{eqn:assumption}
\mathbf{L} = \mathbf{R}^{\top} \mathbf{P}, \text{ or, equivalently, } \mathbf{I} = \mathbf{R}^{\top} \mathbf{P R}.
\end{equation}
In later sections, we assume both that $\mathbf{B}$ is diagonalizable as well as the stricter relation of Equation \eqref{eqn:assumption}, though
fully characterizing either condition remains an open question.
In the rest of this section, we make some progress in characterizing assumption \eqref{eqn:assumption}.

\begin{rem*}
Equation \eqref{eqn:assumption} as well as the second statement in Lemma \ref{lem:left-right} feature the \textit{transpose} of the right eigenvectors of $\mathbf{B}$, and not the conjugate transpose, even if the corresponding eigenvalue is complex.
This is not a mistake.
Indeed, assumption \eqref{eqn:assumption} is equivalent to $\mathbf{R R^{\top}} = \mathbf{P}$.
The more natural relation $\mathbf{R R^*} = \mathbf{P}$ is impossible as it would imply that $\mathbf{P}$ is positive definite, which is false as $\mathbf{P}$ always has eigenvalues $\pm1$. 
Assumption \eqref{eqn:assumption} leads to no such contradiction and is in fact frequently true, as explained below.
\end{rem*}

For any two complex vectors $\mathbf{x},\mathbf{y}$, define the bilinear form $\langle \mathbf{x}, \mathbf{y} \rangle = \mathbf{x^{\top}} \mathbf{P} \mathbf{y}$.
Assuming $\mathbf{B}$ is diagonalizable, 
$\mathbf{R^{\top} P R}$ is the matrix of the bilinear form $\langle \cdot, \cdot \rangle$ in the basis $\mathbf{R}$ and statement 3 of Lemma \ref{lem:left-right} implies it has block-diagonal form.
Moreover, assumption \eqref{eqn:assumption} holds if and only if each eigenspace of $\mathbf{B}$ admits a basis that is orthonormal with respect to $\langle \cdot, \cdot \rangle$.
Two special eigenvalues always admit such a basis.

\begin{lem}\label{lem:orthonormal-bases}
The eigenvalues $\pm1$ of the NB matrix $\mathbf{B}$ admit a basis that is orthonormal with respect to $\langle \cdot, \cdot \rangle$.
\end{lem}

\begin{proof}
Suppose $\mathbf{B v} = - \mathbf{v}$.
Equations \eqref{eqn:pm1} imply $\mathbf{P v} = \mathbf{v}$.
Therefore, the restriction of $\mathbf{P}$ to the eigenspace of eigenvalue $-1$ is equal to the identity.
Thus, any orthonormal basis (in the standard sense) of the eigenspace of $-1$ is also orthonormal with respect to $\langle \cdot, \cdot \rangle$.
Similarly, if $\mathbf{B v} = \mathbf{v}$ then Equations \eqref{eqn:pm1} imply $\mathbf{P v} = -\mathbf{v}$.
Suppose the vectors $\{\mathbf{x}_j\}$ are an orthonormal basis (in the standard sense) of the eigenspace of $+1$.
Then $\{i \mathbf{x}_j\}$ are an orthonormal basis with respect to $\langle \cdot, \cdot \rangle$, where $i^2 = -1$.
\end{proof}

In view of this result, if every eigenvalue other than $\pm1$ is simple then assumption \eqref{eqn:assumption} holds. 
We remark that virtually every graph observed in the course of this and previous studies satisfies this condition.

\section{Adding a new node}\label{sec:addnode}

From now and for the rest of the manuscript, we fix a finite, unweighted, undirected, connected graph $G$ with $n$ nodes and $m$ edges and assume that $G$ has minimum degree at least $2$.
Let $\mathbf{B}$ be the NB matrix of $G$. Perron-Frobenius theory guarantees that $\mathbf{B}$ has a simple eigenvalue equal to its spectral radius. We refer to this eigenvalue as the Perron eigenvalue and denote it by $\lambda_1$.
Construct a new graph $G^{c}$ by adding a new node $c$ of degree $d$ to $G$.
Let $\mathbf{B}^c$ be the NB matrix of $G^c$ and suppose $\lambda_c$ is its Perron eigenvalue.
The main purpose of the present work is to bound the difference $| \lambda_c - \lambda_1|$. 
This setting was already studied in \cite{torres2020} in the applied context of node immunization strategies.
In the present work we improve upon some of the results discussed there as well as establish some more theoretical facts in full rigor.

We now quickly recall some of the necessary results from \cite{torres2020}. Importantly, \cite{torres2020} discusses the case of node removal, while in the present work we focus on node addition. Every result of in \cite{torres2020} is also applicable in the case of node addition, and every result proved herein is also applicable in the case of node removal.

Consider $\mathbf{B}$ and $\mathbf{B}^c$. Note $\mathbf{B}$ is a square matrix of size $2m$ and $\mathbf{B}^{c}$ is a square matrix of size $2m+2d$.
We can write $\mathbf{B}^{c}$ in block form as shown in the bottom right of Figure \ref{fig:add-node}, where $\mathbf{B}$ is the NB matrix of the original graph, and $\mathbf{F}$ is indexed in the rows and columns by the yellow edges.
Accordingly, $\mathbf{D}$ is indexed in the rows by blue edges and in the columns by yellow edges, and vice versa for $\mathbf{E}$.
Note that all of $\mathbf{B},\mathbf{D},\mathbf{E},\mathbf{F}$ are sub-matrices of $\mathbf{B}^{c}$ and thus we know their general element is given by Equation (\ref{eqn:def-b}).

\begin{figure}
\begin{centering}
\includegraphics[scale=0.75]{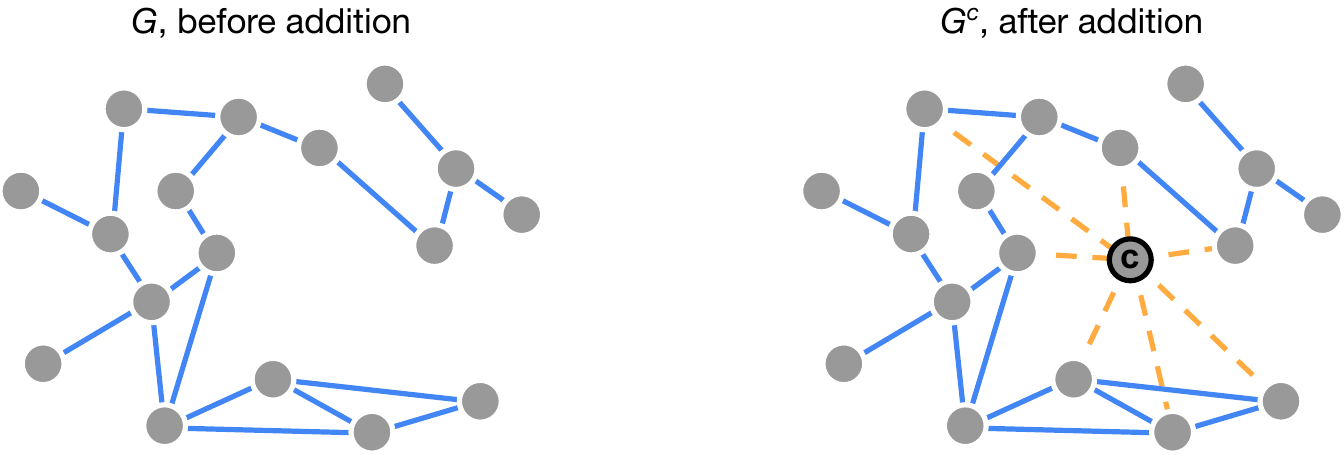}
\par\end{centering}
\begin{flalign*}
\quad\quad\quad\quad\quad
\left(
\begin{array}{ccc}  &   & \\  & \mathbf{B} & \\  &   & \\   \end{array} \right)
\quad\quad\quad\quad\quad\quad\quad\quad\quad
\mathbf{B}^c= \left(
\begin{array}{c;{2pt/2pt}c} \begin{array}{ccc}  &   & \\
& \mathbf{B} & \\  &   & \\   \end{array}  & \mathbf{D} \\  \hdashline[2pt/2pt] \mathbf{E} & \mathbf{F} \\ \end{array}
\right)
\quad\quad\quad
\end{flalign*}

\caption{\textbf{\label{fig:add-node}Top:} Construction of $G^{c}$ from $G$.
All edges incident to $c$ are in dashed yellow lines. 
\textbf{Bottom:} corresponding NB-matrices.
Adapted with permission from \cite{torres2020}.}
\end{figure}

In \cite{torres2020} it was established that $\mathbf{F}^{2}=0$ and $\mathbf{D}\mathbf{E}=0$.
Now define $\mathbf{X} \coloneqq \mathbf{D}\mathbf{F}\mathbf{E}$ and note 
\begin{equation}\label{eqn:def-x}
\mathbf{X}_{k\to l,i\to j} = a_{ck}a_{cj}(1 - \delta_{jk}),
\end{equation}
that is, $\mathbf{X}$ is a binary matrix.
Following the top right panel of Figure \ref{fig:add-node}, $\mathbf{X}$ keeps track of NB walks that consist of four edges of colors blue-yellow-yellow-blue.
These are precisely those paths formed by the addition of the new node $c$ and thus $\mathbf{X}$ will be essential to our discussion.

In \cite{torres2020}, it was also shown, using Schur complements, that the characteristic polynomial of $\mathbf{B}^c$ satisfies
\begin{equation}\label{eqn:schur}
\det\left(\mathbf{B}^{c} - t\mathbf{I}\right) = t^{2d}\det\left(\mathbf{B} - t\mathbf{I} + \frac{\mathbf{X}}{t^{2}}\right),
\end{equation}
for each $t \neq 0$.
Note that $\mathbf{X}$ is the zero matrix when the degree of the added node is $d=1$.
In this case, Equation \eqref{eqn:schur} shows that adding a node of degree one does not affect the non-zero part of the spectrum.
Thus in the following we assume $d \geq 2$.
Now define the \emph{resolvent} of $\mathbf{B}$, $\mathbf{Y}(t) \coloneqq \left(\mathbf{B} - t\mathbf{I}\right)^{-1}$ and factor it out to get
\begin{equation}\label{eqn:yx}
\det\left(\mathbf{B}^{c} - t\mathbf{I}\right) = t^{2d} \det\left(\mathbf{B} - t\mathbf{I}\right) \det\left(\mathbf{I} + \frac{\mathbf{Y}(t) \mathbf{X}}{t^{2}}\right).
\end{equation}

\section{Interlacing of the Perron eigenvalue}\label{sec:interlacing}
In this section, we establish that $\lambda_c > \lambda_1$.
Equation \eqref{eqn:yx} implies that $t$ is an eigenvalue of $\mathbf{B}^{c}$ that is not an eigenvalue of $\mathbf{B}$ if and only if $\det\left( \mathbf{I} + \frac{\mathbf{Y}(t) \mathbf{X}}{t^{2}}\right) = 0$, that is, exactly when $-t^{2}$ is an eigenvalue of $\mathbf{Y}(t) \mathbf{X}$.
In the following lines, we give $\mathbf{Y}(t)$ a suitable form, which will then allow us to show that there exists a real eigenvalue of $\mathbf{B}^{c}$, namely its Perron eigenvalue $\lambda_{c}$, such that $\lambda_{c} > \lambda_1$.
We proceed along the following steps:
\vspace{-1em}
\begin{enumerate}
\item Use the assumption of diagonalizability of $\mathbf{B}$ to rewrite its resolvent $\mathbf{Y}(t)$.
\item Apply the Perron-Frobenius theorem to $\mathbf{Y}(t) \mathbf{X}$ to find its Perron eigenvalue $y(t)$.
\item Define an auxiliary matrix $\mathbf{H} = \mathbf{H}(t)$ which also has $y(t)$ as an eigenvalue.
\item Apply Gershgorin's Disk theorem to $\mathbf{H}$ to show that at some $t_c$ it holds that $y(t_c) = -\left( t_c \right)^{2}$, as desired. By the previous remark, $t_c$ is in fact equal to $\lambda_{c}$, the Perron eigenvalue of $\mathbf{B}^{c}$.
\end{enumerate}

\paragraph*{Step 1: Rewriting the resolvent}
Suppose $\mathbf{B}$ is diagonalizable with $\mathbf{R}$ a matrix of right eigenvectors as columns and $\mathbf{L}$ a matrix with left eigenvectors as rows such that $\mathbf{I} = \mathbf{R}\mathbf{L} = \mathbf{L} \mathbf{R}$.
Define $\mathbf{T}$ as the diagonal matrix with $\mathbf{T}_{ii} = \mathbf{T}_{ii}(t) \coloneqq \left(\lambda_{i} - t\right)^{-1}$ for $i = 1,\ldots,2m$, where $\lambda_{i}$ are the eigenvalues of $\mathbf{B}$ sorted according to decreasing modulus; if two eigenvalues have the same modulus, sort them arbitrarily.
Then we can write 
\begin{equation}
\mathbf{Y}(t) = \left( \mathbf{B} - t\mathbf{I} \right)^{-1} = \left( \mathbf{R} \mathbf{\Lambda} \mathbf{L} - t \mathbf{I} \right)^{-1} = \mathbf{R} \left( \mathbf{\Lambda} - t\mathbf{I}\right)^{-1} \mathbf{L} = \mathbf{R} \mathbf{T} \mathbf{L} = \sum_{i} \frac{\mathbf{v}_{i}^{R}\mathbf{v}_{i}^{L}}{\lambda_{i} - t},
\end{equation}
where $\mathbf{\Lambda}$ contains all eigenvalues in order, the $\mathbf{v}_{i}^{R}$ are the columns of $\mathbf{R}$, and $\mathbf{v}_{i}^{L}$ are the rows of $\mathbf{L}$. 
Note since $\mathbf{R}\mathbf{L} = \mathbf{I}$, we have $\mathbf{v}_{i}^{L}\mathbf{v}_{i}^{R}=1$\@.
As mentioned above, we are looking for an eigenvalue of ${\mathbf{Y}(t) \mathbf{X} = \mathbf{R} \mathbf{T}(t) \mathbf{L} \mathbf{X}}$ that equals $-t^{2}$.
In what follows we drop the dependence on $t$ when possible for ease of notation.

\begin{lem}
[\textbf{Step 2: Apply the Perron-Frobenius theorem}]\label{lem:yx-negative}
Fix a real number $t$ with $t > \lambda_1$ and let $\rho(t)$ be the spectral radius of $\mathbf{Y}(t) \mathbf{X}$. Then $\mathbf{Y}(t) \mathbf{X}$ has a simple real negative eigenvalue $y(t)$ such that $y(t) = -\rho(t)$.
\end{lem}

\begin{proof}
Since $t > \lambda_{1}$, the following Neumann series converges and we have, for any two oriented edges $e_{1},e_{2}$,
\begin{equation}
\left(\mathbf{Y} \mathbf{X} \right){}_{e_{1}e_{2}} = -\sum_{k=0}^{\infty} \frac{1}{t^{k+1}} \left(\mathbf{B}^{k} \mathbf{X}\right)_{e_{1}e_{2}}.
\end{equation}
Since both $\mathbf{B}$ and $\mathbf{X}$ are binary matrices, $( \mathbf{B}^k \mathbf{X} )_{e_1 e_2}$ is non-negative. Furthermore, since the graph is connected, for each entry $e_{1}e_{2}$ there exists an integer $k$ such that $(\mathbf{B}^{k})_{e1e2}$ is positive.
Therefore, $\left( \mathbf{Y} \mathbf{X} \right)_{e_{1}e_{2}}$ is negative unless every element in the $e_{2}$ column of $\mathbf{X}$ is zero.
Thus, after reordering its columns, $\mathbf{Y} \mathbf{X}$ has the block form
\begin{equation}
\mathbf{Y} \mathbf{X}=\left(\begin{array}{cc}
\mathbf{Y}_{1} & 0\\
\mathbf{Y}_{2} & 0
\end{array}\right),
\end{equation}
for some square matrix $\mathbf{Y}_{1}$ and rectangular matrix $\mathbf{Y}_{2}$.
Thus the eigenvalues of $\mathbf{Y} \mathbf{X}$ are equal to the eigenvalues of $\mathbf{Y}_{1}$.
But the entries of $\mathbf{Y}_{1}$ are all strictly negative, thus the Perron-Frobenius theorem implies that there is a negative real number $y = y(t)$ such that it is a simple eigenvalue of $\mathbf{Y} \mathbf{X}$ equal to $-\rho(t)$.
\end{proof}

\paragraph*{Step 3: Define the auxiliary matrix.}
Consider the matrix $\mathbf{H} = \mathbf{H}(t) \coloneqq \mathbf{TLXR}$.
Note that $\mathbf{H}$ and $\mathbf{YX} = \mathbf{RTLX}$ are cyclic permutations of the same matrix product and therefore they have the same eigenvalues.
In particular $y$ is an eigenvalue of $\mathbf{H}$, for each $t$. Furthermore, we have
\begin{equation}
\mathbf{H}_{ij} = \frac{\mathbf{v}^L_i \mathbf{X}  \mathbf{v}^R_j}{\lambda_i - t}.
\end{equation}
For convenience, in what follows we write $\alpha_{ij} \coloneqq \mathbf{v}^L_i \mathbf{X} \mathbf{v}^R_j$. Importantly, $\alpha_{ij}$ is constant with respect to $t$.

\begin{thm}[\bf{Step 4: Apply Gershgorin's Disk theorem}]\label{thm:interlacing}
Suppose $\mathbf{B}$ is diagonalizable with resolvent $\mathbf{Y}(t)$ and let $\mathbf{X}$ be defined by Equation \eqref{eqn:def-x} after adding a new node c of degree $d \geq 2$.
Then there exists a real number $\lambda_{c} > \lambda_1$ such that $-\lambda_{c}^{2}$ is an eigenvalue of $\mathbf{Y}(\lambda_{c}) \mathbf{X}$ and $\lambda_{c}$ is an eigenvalue of $\mathbf{B}^{c}$.
\end{thm}

\begin{proof}
Put $r_{j} \coloneqq \sum_{i \neq j} \left| \mathbf{H}_{ij} \right|$ and define the $j^{th}$ Gershgorin disk as $D_{j}\coloneqq\left\{ z:\left|z - \mathbf{H}_{jj}\right|\leq r_{j}\right\} $.
Note here that both the center and the radius of each disk $D_{j}$ are changing as a function of $t$.
Gershgorin's disk theorem says that all eigenvalues of $\mathbf{H}$ must be contained in the union of all $D_{j}$.
Furthermore, a strengthened version of the theorem says that if one of the disks is isolated from the rest, then it must contain exactly one eigenvalue (see e.g. Theorem 6.1.1 in \cite{horn}).

Now let $\epsilon \coloneqq t - \lambda_1$. To prove the existence of $\lambda_{c}$, we proceed to prove the following three assertions, as illustrated in Figure \ref{fig:disks}:
\vspace{-0.5em}
\begin{enumerate}[leftmargin=.6cm]
\item[(1)] For small $\epsilon > 0$, $D_{1}$ is disjoint from all other disks and it must contain $y$, the least eigenvalue of $\mathbf{H}$. 
\item[(2)] For small $\epsilon > 0$, every real number in $D_{1}$ is less than $-t^{2}$.
\item[(3)] As $\epsilon$ goes to $\infty$, every real number in $D_{j}$ must be greater than $-t^{2}$, for all $j$.
\end{enumerate}
\vspace{-0.5em}
Since $y$ is a real continuous function of $t$, these three assertions imply that at some point $\lambda_{c}$, with $ \lambda_c > \lambda_1$, we must have $y(\lambda_{c})=-\lambda_{c}^{2}$, and therefore the theorem follows. We address all three claims in turn with the following inequalities. Note we can write
\begin{equation}
D_{j} = \left\{ z: \left|z - \mathbf{H}_{jj} \right| < r_{j} \right\} = \left\{ z:\left|z - \frac{\alpha_{jj}}{\lambda_{j} - \lambda_1 - \epsilon} \right| \le \sum_{i\neq j} \left| \frac{\alpha_{ij}}{\lambda_i - \lambda_{1} - \epsilon} \right| \right\} .
\end{equation}

\begin{figure}
\begin{centering}
\includegraphics[width=\textwidth]{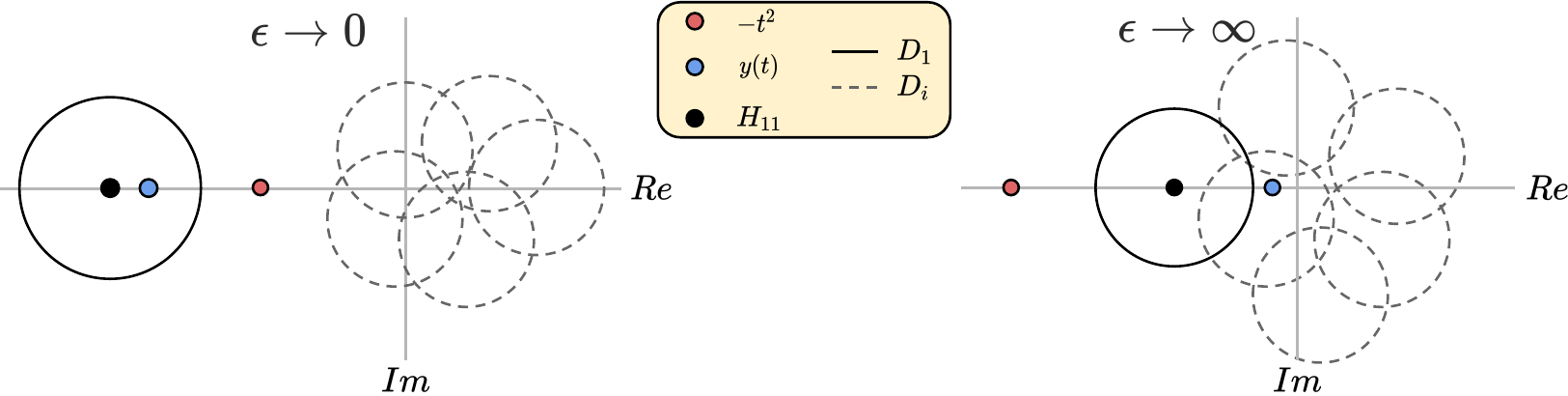}
\par\end{centering}
\caption{\label{fig:disks}
\textbf{Left: } If  $t\!=\!\lambda\!+\!\epsilon$ and $\epsilon$ is small,
$y(t)$ lies inside $D_{1}$ which in turn lies to the left of $-t^{2}$
and is disjoint from the rest. \textbf{Right:} when $\epsilon \to \infty$,
$y(t)$ lies in some of the $D_{i}$, all of which lie to the right
of $-t^{2}$.}
\end{figure}

For assertion (1), consider $D_{1}$ when $\epsilon$ approaches $0$ from above.
Write $\mathbf{H}_{11} + \eta r_{1}$ for an arbitrary number inside $D_{1}$ where $\eta$ is a complex number with $\left|\eta\right|\le1$.
Similarly, write $\mathbf{H}_{jj}+\eta'r_{j}$ for an arbitrary element in $D_{j}$, for $j\neq1$, where $\eta'$ is a complex number with $|\eta'|\le1$.
Then we have 
\begin{equation}
\left| \mathbf{H}_{11} + \eta r_{1} \right| = \left| - \frac{\alpha_{11}}{\epsilon} + \eta \sum_{i \neq 1} \left| \frac{\alpha_{i1}}{\lambda_i - \lambda_1 - \epsilon} \right| \right| > \left| \frac{\alpha_{jj}}{\lambda_{j} - \lambda_1 - \epsilon} + \eta' \sum_{i \neq j} \left| \frac{\alpha_{ij}}{\lambda_i - \lambda_1 - \epsilon} \right| \right| = \left|\mathbf{H}_{jj} + \eta'r_{j} \right|.
\end{equation}
When $\epsilon$ tends to zero from above, the left-hand side of the inequality tends to infinity, regardless of the value of $\eta$ or $\alpha_{ij}$.
The right-hand side of the inequality tends to a constant, regardless of the other values.
Thus, the inequality holds for small enough $\epsilon$, and in this regime $D_{1}$ is disjoint from all other disks.
Since $y$ is the least eigenvalue of $\mathbf{H}$, $D_{1}$ contains $y$ and no other eigenvalue of $\mathbf{H}$.

For assertion (2), consider an arbitrary real number inside $D_{1},$ written as $\mathbf{H}_{11}+\eta r_{1}$ for some real $\eta\in[-1,1]$. The following inequality holds when $\epsilon$ is sufficiently small,
\begin{equation}
\mathbf{H}_{11} + \eta r_{1} = - \frac{\alpha_{11}}{\epsilon} + \eta \sum_{i \neq 1} \left| \frac{\alpha_{i1}}{\lambda_i - \lambda_1 - \epsilon} \right| < -\left( \lambda_1 + \epsilon \right)^{2},
\end{equation}
provided that $\alpha_{11}$ is non-negative.
But applying the Perron-Frobenius theorem to $\mathbf{B}$ implies that $\mathbf{v}_{1}^{R}$ and $\mathbf{v}_{1}^{L}$ can be chosen with all positive entries and therefore $\alpha_{11}=\mathbf{v}_{1}^{L} \mathbf{X} \mathbf{v}_{1}^{R}$ is positive.

For assertion (3), when $\epsilon \to \infty$ we have 
\begin{equation}
-\left( \lambda_1 + \epsilon \right)^{2} < \frac{\alpha_{jj}}{\lambda_j - \lambda_1 - \epsilon} + \eta \sum_{i \neq j} \left| \frac{\alpha_{ij}}{\lambda_{j} - \lambda_1 - \epsilon} \right|,
\end{equation}
for each $j$ and each real $\eta\in[-1,1]$. This finishes the proof.
\end{proof}

\begin{rem*}
The coefficient $\alpha_{11} = \mathbf{v}^L_1 \mathbf{X} \mathbf{v}^R_1$ is what the authors of \cite{torres2020} called the \emph{X-non-backtracking centrality} of the newly added node $c$.
In the context of the cited paper, it was used to quantify the influence of the node $c$ in the Perron eigenvalue.
\end{rem*}

\begin{rem*}
The arguments in this and following sections apply without modification to the inverse setting of node removal.
That is, when the starting graph is $G^c$ and we construct $G$ by removing an existing node $c$. 
Thus, the main theorem shows that the Perron eigenvalue never decreases when adding a node and it never increases when removing a node. 
Furthermore, as explained before, nodes of degree $1$ have no influence in the non-zero part of the spectrum.
Thus, when the degree of $c$ is at least $2$, the Perron eigenvalue always increases when $c$ is added and it always decreases when $c$ is removed.
\end{rem*}

\section{Upper bounds}\label{sec:bounds}
Theorem \ref{thm:interlacing} allows us to bound the change in the Perron eigenvalue after node addition.

\begin{thm}\label{thm:bound}
Let $\lambda_c = \lambda_1 + \epsilon_c$ be the Perron eigenvalue after node addition.
Define $\gamma$ as the spectral gap of $\mathbf{B}$, that is $\gamma \coloneqq \max_i |\lambda_1 - \lambda_i|$, and let $\| \cdot \|_p$ be any induced $p$-norm for $1 \leq p \leq \infty$.
Then,
\begin{equation}\label{eqn:bound-1}
\epsilon_c \leq \max \left\{ \frac{\left\| \mathbf{L X R} \right\|_p}{\lambda_1^2}, \sqrt{\frac{\left\| \mathbf{L X R} \right\|_p}{\gamma}} - \lambda_1 \right\}.
\end{equation}
\end{thm}

\begin{proof}
Recall the norm of a matrix is greater than or equal to its spectral radius, for any sub-multiplicative norm (see e.g. Theorem 5.6.9 of \cite{horn}). In particular, we have
\begin{equation}\label{eqn:norm-y}
\left| y \left( t \right) \right| \leq \left\| \mathbf{H} \right\|_p \leq \left\| \mathbf{T} \right\|_p \left\| \mathbf{L X R} \right\|_p.
\end{equation}
Now put $\epsilon = t - \lambda_1 > 0$ for an arbitrary $t$ and consider the norm of $\mathbf{T}$.
For any $1 \leq p \leq \infty$, we have
\begin{equation}\label{eqn:norm-t}
\left\| \mathbf{T} \right\|_p = \max_j \frac{1}{|\lambda_j - t|} = \max \left\{ \frac{1}{\left| \lambda_1 - t \right|}, \,\, \max_{j \neq 1} \frac{1}{|\lambda_j - t|} \right\} \leq \max \left\{ \epsilon^{-1}, \gamma^{-1} \right\}.
\end{equation}
We now consider Equations \eqref{eqn:norm-y} and \eqref{eqn:norm-t} in the particular case  $\epsilon = \epsilon_c$, that is, the time $t_c$ at which $t_c = \lambda_c = \lambda_1 + \epsilon_c$ and $y\left( t_c \right) = - \left( t_c \right)^2 $.
We consider two cases, depending on the maximum in Equation \eqref{eqn:norm-t}.

First, if $\left| y(t_c) \right| \leq \epsilon_c^{-1} \left\| \mathbf{L X R} \right\|_p$, we have
\begin{equation}
\lambda_1^2 < \left| -\left( \lambda_1 + \epsilon_c \right)^2 \right| = \left| y \left( t_c \right) \right| \leq \left( \epsilon_c \right)^{-1} \left\| \mathbf{L X R} \right\|_p,
\end{equation}
and thus
\begin{equation}\label{eqn:ineq-1}
\epsilon_c \leq \left\| \mathbf{L X R} \right\|_p / \lambda_1^2.
\end{equation}

Second, if $\left| y(t_c) \right| \leq \gamma^{-1} \left\| \mathbf{L X R} \right\|_p$, we have
\begin{equation}
\lambda_1 + \epsilon_c = \sqrt{\left| y\left( t_c \right) \right| } \leq \sqrt{ \gamma^{-1} \left\| \mathbf{L X R} \right\|_p },
\end{equation}
which finishes the proof.
\end{proof}

The upper bound in \eqref{eqn:bound-1} is not very satisfying because it still depends on $\| \mathbf{L X R} \|_p$, which requires knowledge of the entire eigendecomposition of $\mathbf{B}$.
We now provide another bound, for the case $p = 2$, that only requires knowledge of elementary information.




\begin{prop}\label{pro:bound-x}
Let $\mathbf{B}$ be diagonalized as $\mathbf{B} = \mathbf{R \Lambda L}$ and let $\mathbf{X}$ be defined as in Equation \eqref{eqn:def-x} by adding node $c$. Let $\deg \left(j\right)$ be the degree of node $j$ before $c$ is added to the graph. Then, assuming that Equation \eqref{eqn:assumption} holds, we have
\begin{equation}\label{eqn:main-bound}
\left\| \mathbf{L X R} \right\|_2 \leq \mathbf{1}^{\top} \mathbf{X} \mathbf{1} = \left( \sum_j a_{cj} \deg \left( j \right) \right)^2 - \sum_j a_{cj} \deg\left(j \right)^2,
\end{equation}
where $n$ is the number of nodes in the graph.
\end{prop}

\begin{rem*}
Recall $\mathbf{B}$ is not normal and thus $\mathbf{R}, \mathbf{L}$ are not unitary.
Thus, some usual properties of norms do not apply.
For instance, it is not true that $\| \mathbf{L X R} \|_2 = \| \mathbf{X} \|_2$.
For this reason, the inequality in Equation \eqref{eqn:main-bound} needs to be carefully considered.
Indeed, the proof of Proposition \ref{pro:bound-x} relies heavily on the assumption that $\mathbf{L X R}$ is complex symmetric.
We direct readers interested in the theory of complex symmetric matrices to \cite{horn,garcia2014}.
\end{rem*}

\begin{proof}
The identity $\left\| \mathbf{X} \right\|_1 = \left( \sum_j a_{cj} \deg \left( j \right) \right)^2 - \sum_j a_{cj} \deg\left(j \right)^2$ is proved directly from Equation \eqref{eqn:def-x}. Thus, all we need to show is $\left\| \mathbf{L X R} \right\|_2 \leq \mathbf{1}^{\top} \mathbf{X} \mathbf{1}$.

%
%
%

Assuming Equation \eqref{eqn:assumption} holds, $\mathbf{L X R}$ is complex symmetric (but not Hermitian).
According to the theory of complex symmetric matrices, by Theorem 3.11 of \cite{garcia2014}, we have
\begin{equation}\label{eqn:singular}
\| \mathbf{L X R} \|_2 = \max_{\| \mathbf{x} \| = 1} \Re \left( \mathbf{x}^{\top} \mathbf{L X R} \mathbf{x} \right).
\end{equation}
Let $\{\mathbf{z}_i \}$ be the columns of $\mathbf{L}$.
The vectors $\mathbf{z}_i$ form a basis, though they need not be (left or right) eigenvectors.
Now let $\mathbf{x}$ be a vector with $\| \mathbf{x} \| = 1$ and write it in the chosen basis as $\mathbf{x} = \sum_i x_i \mathbf{z}_i$.
Under our assumptions, we have $\mathbf{I} = \mathbf{RL}$ and $\mathbf{L} = \mathbf{R^{\top} P}$ and thus 
\begin{equation}\label{eqn:real-part}
\Re \left( \mathbf{x}^{\top} \mathbf{L X R} \mathbf{x} \right) = \Re \sum_{i,j} x_i x_j \, \mathbf{z}_i^{\top} \mathbf{R^{\top} P X R} \mathbf{z}_j = \sum_{i,j} \Re \left( x_i x_j \right) (\mathbf{P X})_{ij} \leq \sum_{i,j} (\mathbf{P X})_{ij} = \sum_{i,j} \mathbf{X}_{ij} = \mathbf{1}^{\top} \mathbf{X 1},
\end{equation}
where the inequality is true since $\| \mathbf{x} \| = 1$ and we have used that $\mathbf{P}$ is a permutation matrix.
\end{proof}

\begin{rem*}
The number $\mathbf{1}^{\top} \mathbf{X} \mathbf{1}$ is what the authors of \cite{torres2020} called the \emph{X-degree centrality} of the newly added node $c$. In the cited paper it was argued that $\mathbf{1}^{\top} \mathbf{X} \mathbf{1}$ has a similar behavior to $\alpha_{11} = \mathbf{v}^{L}_1 \mathbf{X} \mathbf{v}^R_1$.
\end{rem*}

\begin{rem*}
The bounds explored in this work leave some room for improvement.
However, tighter bounds for $\left\| \mathbf{L X R} \right\|_p$, for any $p$, have proven to be elusive in the general case.
Nevertheless, it is our experimental observation that $y\left( t_c \right)$ is actually much closer to $\alpha_{11}$ than the present methods would lead one to believe.
In fact, in a computational setting the approximation $\epsilon_c \approx \alpha_{11} / \lambda_1^2$ has been used (cf. Equation \eqref{eqn:ineq-1}) with an empirical absolute error of up to $10^{-4}$. See Section 3.2 and Figure 3 in \cite{torres2020}.
Further progress in this area should focus on establishing this latter approximation in full rigor, as well as on finding tight bounds for $\alpha_{11}$.
\end{rem*}


\section{Conclusion}\label{sec:conclusion}
We have established a weak version of eigenvalue interlacing for the NB-matrix.
Indeed, after adding (or removing) the rows and columns incident to the same node $c$, the Perron eigenvalue behaves as expected: it can only increase when a new node is added to the graph, and it can only decrease when a node is removed from the graph.
However, the other eigenvalues do not seem to behave similarly.
It remains an open question if more general versions of interlacing apply
to the NB matrix.
Importantly, the assumption of diagonalizability is essential in the derivation of Theorems \ref{thm:interlacing} and \ref{thm:bound}.
We reiterate here that virtually all observed graphs in our experience have a diagonalizable NB matrix.
Furthermore, assuming that Equation \eqref{eqn:assumption} holds, we derive a bound in Propositoin \ref{pro:bound-x} that can be computed using only elementary information about the graph.
This latter assumption is tantamount to assuming that each eigenspace of the NB matrix corresponding to eigenvalues other than $\pm 1$ admits a basis that is orthogonal with respect to the inner product defined by the orientation reversal operator $\mathbf{P}$.
This property holds when all eigenvalues other than $\pm 1$ are simple.
Once again, virtually all graphs studied satisfy this property.
Still, a formal characterization of this property is a work in progress.

\section*{Acknowledgements}

This work started while the author was at Northeastern University’s Network Science Institute and supported in part by NSF IIS-1741197.
The author thanks Gabor Lippner and Tina Eliassi-Rad for many invaluable conversations.

\end{document}